\documentclass{amsart}

\usepackage[centertags]{amsmath}
\usepackage{amssymb,amsthm, mathrsfs}
\usepackage[all]{xy}
\usepackage{hyperref}
\usepackage{graphicx}
\usepackage{color}

\title[Parametrized cobordism categories \dots ]{Parametrized cobordism categories and the Dwyer--Weiss--Williams 
index theorem}
\author{George Raptis}
\address{Fakult\"{a}t f\"ur Mathematik \\
Universit\"{a}t Regensburg \\
93040 Regensburg, Germany}
\email{georgios.raptis@ur.de}
\author{Wolfgang Steimle}
\address{Universit\"at Augburg\\
               Institut f\"ur Mathematik\\
               D-86135 Augsburg, Germany}
\email{wolfgang.steimle@math.uni-augsburg.de}


\DeclareMathOperator{\colim}{colim}

\DeclareMathOperator{\fib}{fib}

\DeclareMathOperator*{\hocolim}{hocolim}

\DeclareMathOperator*{\holim}{holim}

\DeclareMathOperator{\id}{id}

\DeclareMathOperator{\map}{map}

\DeclareMathOperator{\simp}{simp}

\newcommand{\sections}[2]{\Gamma\biggl(\begin{array}{c} {#1}\\ \downarrow\\ {#2}\end{array} \biggr)}

\newcommand{\Gr}{\mathrm{Gr}}

\newcommand{\cat}{\mathbf{Cat}}

\newcommand{\op}{^{op}}
\newcommand{\inv}{^{-1}}
\newcommand{\fibr}[2]{\begin{pmatrix} {#1} \\ \downarrow \\ {#2}\end{pmatrix}}

\newcommand{\RR}{\mathbb{R}}

\newcommand{\calU}{\mathcal{U}}

\newcommand{\calC}{\mathcal{C}}

\newcommand{\calA}{\mathcal{A}}

\newcommand{\dell}{\partial}

\newcommand{\Rfd}{\mathcal{R}^{hf}}

\newcommand{\tr}{\operatorname{PT}}
\newcommand{\BGtr}{\operatorname{tr}}
\newcommand{\Path}{\operatorname{Path}}
\newcommand{\univ}{\mathrm{univ}}

\newcounter{commentcounter}

\newcommand{\addws}[1]{#1}

\begin{document}

\theoremstyle{plain}
\newtheorem{thm}{Theorem}
\newtheorem{cor}[thm]{Corollary}
\newtheorem{lem}[thm]{Lemma}
\newtheorem{prop}[thm]{Proposition}
\newtheorem{claim}[thm]{Claim}
\theoremstyle{definition}
\newtheorem{defn}[thm]{Definition}
\newtheorem{obs}[thm]{Observation}
\newtheorem{constr}[thm]{Construction}
\newtheorem*{indextheorem}{Index Theorem}

\theoremstyle{remark}
\newtheorem{rem}[thm]{Remark}
\newtheorem*{rem*}{Remark}
\newtheorem*{ex*}{Example}

\setcounter{secnumdepth}{2}
\numberwithin{thm}{section}

\SelectTips{eu}{10}
\renewcommand{\theenumi}{\roman{enumi}}
\renewcommand{\labelenumi}{\textup{(}\theenumi\textup{)}}
\renewcommand{\theequation}{\arabic{equation}}

\begin{abstract}
We define parametrized cobordism categories and study their formal properties as bivariant theories. Bivariant transformations to a strongly excisive bivariant theory give rise to 
characteristic classes of smooth bundles with strong additivity properties. In the case of cobordisms between manifolds with boundary, we prove that such a bivariant transformation
is uniquely determined by its value at the universal disk bundle. This description of bivariant transformations yields a short proof of the Dwyer-Weiss-Williams family index theorem for the parametrized 
$A$-theory Euler characteristic of a smooth bundle. 
\end{abstract}

\maketitle


\section{Introduction}

Let $M$ be a compact smooth manifold embedded in $\RR^N$. A classical theorem \cite[Theorem 2.4]{Becker-Gottlieb(1975)} says that the Euler characteristic of $M$ is equal to 
the mapping degree $d(M)$ of the map 
$$S^N\to S^N$$ 
that collapses $x \in S^N =\RR^N\cup \{\infty\}$ to $\infty$ if $x$ lies outside an $\varepsilon$-neighborhood of $M$, 
and sends it to the difference vector (suitably scaled) from $x$ to its closest point in $M$, otherwise. 

This can also be stated as follows. The $A$-theory characteristic of a compact manifold $M$ is a refinement of the Euler characteristic of $M$ to a point $\chi(M)\in A(M)$, Waldhausen's algebraic $K$-theory of the space $M$ \cite{Waldhausen(1985)}. More precisely, and assuming 
for simplicity that $M$ is connected, the path component of $\chi(M)$ in $\pi_0 A(M)$ corresponds under the canonical isomorphism
\[\pi_0 A(M)\cong\mathbb Z,\]
to the classical Euler characteristic \addws{of $M$}. Moreover, the degree of the collapse map of $M$ described above refines to an element $\BGtr(M)\in Q(M_+)$, the stable homotopy of $M$ 
with a disjoint base-point. Then, for connected $M$, the theorem says that under the canonical isomorphism
\[\pi_0 Q(M_+)\cong\mathbb Z,\]
the classes of $\BGtr(M)$ and $\chi(M)$ agree.

These refined invariants can be naturally extended to interesting invariants of \addws{fiber} bundles of compact manifolds. A \addws{fiber} bundle of compact smooth manifolds $p: E \to B$
has a parametrized $A$-theory characteristic and a parametrized transfer map
\[\chi(p)\in \sections{A_B(E)}{B},\quad \BGtr(p)\in\sections{(Q_+)_B(E)}{B}\]
which are sections of associated fibrations over $B$ whose fibers at $b\in B$ are the spaces $A(p\inv(b))$ and $Q(p\inv (b)_+)$, respectively. Moreover, there is a 
``unit'' map $\eta\colon Q(M_+)\to A(M)$ which is the natural infinite loop space map that is determined by the fact that it sends $1\in \pi_0 Q(S^0)=\mathbb Z$ 
to $1\in \pi_0 A(\{*\})=\mathbb Z$. 

The Dwyer-Weiss-Williams smooth index theorem \cite{Dwyer-Weiss-Williams(2003)} 
is a family version of the aforementioned classical result that identifies the parametrized $A$-theory characteristic of a 
\addws{fiber bundle of compact smooth manifolds} with the composition of the 
parametrized transfer map followed by the (parametrized) unit map. 

\begin{indextheorem}[{\cite[Theorem 8.5]{Dwyer-Weiss-Williams(2003)}}]
Let $p: E \to B$ be a bundle of compact smooth manifolds. Then $\chi(p): B \to A_B(E)$ is (fiberwise) homotopic to the composite map 
$B \xrightarrow{\BGtr(p)} (Q_+)_B(E) \xrightarrow{\eta} A_B(E)$.
\end{indextheorem}

The proof in \cite{Dwyer-Weiss-Williams(2003)} is quite intricate. Firstly, it is obtained from an analogous theorem (also in \cite{Dwyer-Weiss-Williams(2003)}) 
for bundles of compact topological manifolds, whose proof involves constructions of controlled algebraic $K$-theory spectra in order to model the assembly map for 
$A$-theory. Secondly, this topological version of the index theorem is based on a comparison between specific characteristic classes for topological euclidean 
$n$-bundles. 

\addws{
In this note we give a new proof of the Index Theorem which only involves the smooth category.
The main idea is to make systematic use of the fact that both characteristic classes $\chi(p)$ and $\eta \circ \BGtr(p)$ satisfy very strong naturality and additivity properties. 

More rigorously, we introduce \emph{parametrized cobordism categories} (of manifold bundles with boundaries) which, for purely formal reasons, give rise to a \emph{bivariant theory} in the sense that there are covariant and contravariant functorial operations which are compatible with each other. It turns out that both characteristic classes $\chi(p)$ and $\BGtr(p)$ extend to \emph{bivariant transformations} out of this newly defined bivariant theory; this is supposed to reflect the naturality and additivity properties of these two constructions. 

In Theorem \ref{thm:factorization_result} we prove a stronger version of the Index Theorem which asserts that the bivariant transformations extending $\chi(p)$ and $\eta\circ \BGtr(p)$ agree in the homotopy category of bivariant theories. Informally, this means that, there are homotopies 
$$\chi(p)\simeq \eta\circ \BGtr(p)$$ 
as in the formulation of the Index Theorem above which are in addition compatible with the naturality and additivity properties of these two characteristic classes.

It turns out that this stronger version of the Index Theorem has a comparatively simple proof. Indeed, we show in Theorem \ref{thm:mapping_property} that two bivariant transformations out of the bivariant cobordism category agree in the homotopy category provided only that the corresponding characteristic classes agree on \emph{linear disk bundles}  -- this is a formal consequence of a theorem of Genauer \cite{Genauer(2008)} on the homotopy type of the cobordism 
category of manifolds with boundary. But both our characteristic classes can be easily computed on disk bundles and the computation shows that they indeed agree.

\medskip 

The paper is organized as follows. In Section 2, we define parametrized cobordism categories and discuss their properties as bivariant theories. In Section 3, we state the \emph{mapping property} of the parametrized cobordism category with boundaries (Theorem \ref{thm:mapping_property}); this is the classification result for bivariant transformations mentioned above. At the beginning of Section 4, we discuss a generalization of the coassembly construction to bivariant theories and its formal properties. Then we give a description of the coassembly map for the bivariant cobordism category of manifolds with boundary, in terms of a parametrized Pontryagin-Thom construction, and use this to prove the mapping property. In Section 5, we define a bivariant transformation which extends $\chi(p)$ to the bivariant cobordism category, generalizing the construction of maps defined in \cite{Boekstedt-Madsen(2014), Raptis-Steimle(2014), Tillmann(1999)}. Then we prove our main result, the Index Theorem for this bivariant transformation (Theorem \ref{thm:factorization_result}), and deduce the Dwyer-Weiss-Williams Index Theorem in the form that was stated above. 
}

\section{Parametrized cobordism categories}

Let $B$ be a space. We write $BO(d)=\Gr_d(\RR^{d+\infty})$ for the Grassmannian of $d$-dimensional linear subspaces 
in $\RR^{d+\infty}$ \addws{where $\RR^{d + \infty} \colon = \colim_{n \to \infty} \RR^{d + n}$}. We call a fibration 
\[\theta\colon X\to BO(d)\times B\]
a \emph{parametrized $d$-dimensional 
tangential structure} over $B$. 

There is a category $\calC(\theta)$ of parametri\-zed $\theta$-cobordisms 
over $B$ defined as follows. An object in $\calC(\theta)$ is given by a quadruple $(E, p, a, l)$ where:
\begin{itemize}
\item[(i)] $a\in\RR$,
\item[(ii)] $p\colon E\to B$ is a fiber bundle of smooth closed $(d-1)$-dimensional manifolds, which is fiberwise \addws{smoothly}
embedded in $B\times \{a\}\times \RR^{d-1+\infty}$,
\item[(iii)] $l$ is tangential $\theta$-structure, i.e., a lift in the following diagram:
\[\xymatrix{
 && X \ar[d]^\theta\\
 E \ar[rr]_(0.4){(\epsilon\oplus T^vE,p)} \ar@{.>}[rru]^l && BO(d)\times B
}\]
where  $\epsilon\oplus T^v E$ is the once-stabilized fiberwise Gau\ss{} map to the Grassmannian $BO(d)$, classifying the once-stabilized vertical tangent bundle \addws{of $p \colon E \to B$.}
\end{itemize}
\addws{When $B$ is not compact, a fiberwise \emph{smooth} embedding as in (ii) means that the fiberwise embedding of $E$ restricts to fiberwise smooth embeddings in $B \times \RR^{d-1 + N}$, 
for some $N > 0$, over each compact subset of $B$. We refer the reader to \cite{Crabb-James} for general background material on fiberwise differential topology.} A morphism in $\calC(\theta)$ consists of a fiber bundle of compact smooth $d$-manifolds,
$$p: W\to B,$$ 
embedded fiberwise in $B\times [a_0,a_1]\times \RR^{d-1+\infty}$ and cylindrically near $a_0$ and $a_1$, together with 
a tangential $\theta$-structure $l_W\colon W\to X$ which lifts $p$ and the fiberwise Gau\ss{} map along $\theta$. The domain and target of this morphism are the intersections $W_0$ and $W_1$ of $W$ with $\{a_0\}\times\RR^{d-1+\infty}$ and 
$\{a_1\}\times \RR^{d-1+\infty}$ respectively, together with the restrictions of $l_W$ to these subsets. This is well-defined because the once-stabilized Gau\ss{} maps of $W_0$ and $W_1$ are precisely the restrictions of the Gau\ss{} map of $W$, since $W$ is embedded cylindrically near the boundary. 
Thus, we obtain a (non-unital) category where the composition of morphisms is given by union of subsets in $B \times \RR^{d + \infty}$. 

One could give this category a topology using the usual methods (see ~\cite{GMTW(2009)}) but we find it easier to extend $\calC(\theta)$ to a simplicial category, that is, 
a simplicial object in the category $\cat$ of (small, non-unital) categories. To do this, we first discuss the naturality properties of $\calC(\theta)$ with respect 
to $\theta$.

Given a pull-back diagram
\begin{equation}\label{eq:pullback_square}
 \xymatrix{
 X' \ar[rr] \ar[d]^{\theta'} && X \ar[d]^\theta\\
 BO(d) \times B' \ar[rr]^{\id \times f} &&  BO(d) \times B
}
\end{equation}
we get an induced \emph{pull-back functor}
\[f^*\colon \calC(\theta)\to\calC(\theta')\]
which is defined by taking pull-backs of bundles along $f$, and using the pull-back property of 
\eqref{eq:pullback_square} to define the required tangential structures. 

On the other hand, if $g\colon \theta\to \eta$  is a fiberwise map of fibrations over $BO(d)\times B$, then post-composing the tangential $\theta$-structure
with $g$ defines a \emph{push-forward functor}
\[g_*\colon \calC(\theta) \to \calC(\eta).\]
The definitions of $f^*$ and $g_*$ are clearly functorial and commute with each other\footnote{The careful reader may object that the definition of $f^*$ depends not only on $f$, 
but also on the choices of pull-backs. We resolve this issue by requiring that $X$ is a subset of $BO(d)\times B\times \calU$ where $\calU$ is a fixed set of high cardinality and the map to $BO(d) \times B$ is the projection. Then the total space of $\theta'$ may be canonically defined as a subset of $BO(d)\times B'\times \calU$. This also guarantees that the covariant and the contravariant operations commute with each other.}.  
We summarize this by saying that the rule $$\calC \colon \theta\mapsto \calC(\theta)$$ is a \emph{bivariant theory} with values in $\cat$.

\begin{rem}
We will only consider bivariant theories which are defined on the class of fibrations $X \to BO(d) \times B$, for fixed $d$, but one can extend the notion to a more general context. A general abstract notion of a bivariant theory was introduced in \cite{Fulton-MacPherson(1981)}. In that setting it is also required that the theory has product operations, 
which we do not have nor need in our context. 

\end{rem}

To each bivariant theory $F: \theta \mapsto F(\theta)$ \addws{with values in a category $\calA$}, we can formally associate a \emph{simplicial thickening}, denoted $F_{\bullet}$. This is a new bivariant theory with values in the category $\calA^{\Delta^{\op}}$ of simplicial objects in $\calA$. The 
value of the simplicial object $F(\theta)_\bullet$ at 
$[n]$ is given by
\[F(\theta)_n:=F(\theta\times \id_{\Delta^n}),\]
where 
\[\theta\times\id_{\Delta^n}\colon X\times \Delta^n\to BO(d)\times B\times \Delta^n\]
is the fibration pulled back from $\theta$ using the projection map $f\colon B\times \Delta^n\to B$. A simplicial operator $[n]\to [m]$ induces a pull-back diagram of 
fibrations and hence a functor $F(\theta)_m\to F(\theta)_n$. The pull-back and push-forward operations of $F$ clearly extend to 
$F_\bullet$ and give $F_\bullet$ the structure  of a bivariant theory with values in \addws{$\calA^{\Delta^{\op}}$}.

\medskip 

\addws{Applying this construction to the bivariant theory $\calC \colon \theta \mapsto \calC(\theta)$, we obtain our 
model for the parametrized $\theta$-cobordism category and its classifying space. Recall that the classifying space 
of a non-unital category is the geometric realization of its nerve regarded as a semi-simplicial set.}

\addws{
\begin{defn}
We call the simplicial object $\calC(\theta)_{\bullet} \in \cat^{\Delta^{\op}}$ the \emph{parametrized $\theta$-cobordism category}. The \emph{classifying space} of $\calC(\theta)_{\bullet}$, denoted $B \calC(\theta)_{\bullet}$, is the (fat) geometric realization of the degree-wise classifying spaces, 
$$B \calC(\theta)_{\bullet} \colon = \vert [n] \mapsto B(\calC(\theta)_n) \vert.$$
\end{defn}
}

\addws{
Thus the rule 
$B \calC_{\bullet}: \theta \mapsto B \calC(\theta)_{\bullet}$ defines a bivariant theory with values in the category of spaces.} In the case where $B$ is the one-point 
space and $\theta\colon  X \to BO(d)$ is a fibration, we recover a simplicial version of the $\theta$-cobordism category (with discrete cuts) from \cite{GMTW(2009)}. 
Our model in this case is closely related to the sheaf model for the cobordism category as defined in \cite[2.3]{GMTW(2009)}. A related notion of a parametrized 
cobordism category was also considered in a different context in \cite{Stolz-Teichner(2011)}.

\medskip 

By construction, the pull-back and push-forward operations on $\calC(\theta)_{\bullet}$ are enriched over simplicial sets in the following ways: 
\begin{itemize}
\item[(i)] for two spaces over $B$, 
$$f'\colon  B' \to B, \ \ f''\colon  B'' \to B,$$
and a continuous map $h: B''\times \Delta^n\to B'$ of spaces over $B$, there is a natural simplicial functor
\[\calC(\theta')_{\bullet} \times \Delta^n_{\bullet} \to \calC(\theta'')_{\bullet}\]
where $\theta'$, $\theta''$ denote the fibrations over $BO(d) \times B'$ and $BO(d) \times B''$ respectively, which 
are pulled back from the parametrized tangential structure $\theta: X \to BO(d) \times B$ over $B$. This is defined in simplicial 
degree $k$ by
\[
\calC(\theta')_k \xrightarrow{h^*} \calC(\theta'' \times \mathrm{id}_{\Delta^n})_k \xrightarrow{(\mathrm{id}_{B''} \times (|\alpha|, \mathrm{id}_{\Delta^k}))^*} \calC(\theta'')_k
\]
for each $k$-simplex $\alpha: \Delta^k_{\bullet} \to \Delta^n_{\bullet}$.
\item[(ii)] for two parametrized tangential structures over $B$, 
$$\theta\colon X \to BO(d) \times B, \ \ \eta: Y \to BO(d) \times B,$$
and a continuous map $h: X \times \Delta^n\to Y$ over $BO(d)\times B$, there is a simplicial functor
\[\calC(\theta)_{\bullet} \times \Delta^n_{\bullet} \to \calC(\eta)_{\bullet}\]
which is defined in simplicial degree $k$ by 
\[
\calC(\theta)_k \xrightarrow{(\mathrm{id}_X \times (|\alpha|, \mathrm{id}_{\Delta^k}))_*}  \calC(\theta \times \Delta^n)_k
\xrightarrow{h_*} \calC(\eta)_k
 \]
for each $k$-simplex $\alpha: \Delta^k_{\bullet} \to \Delta^n_{\bullet}$. 
\end{itemize}

As a consequence of this enrichment, the bivariant theory $\theta \mapsto B \calC(\theta)_{\bullet}$ preserves fiberwise homotopies, 
contravariantly over $B$ and covariantly over $BO(d) \times B$, for any $B$. 

\begin{defn}
A bivariant theory with values in spaces is \emph{homotopy invariant} if the following hold: 
\begin{enumerate}
\item[(a)] if $g\colon \theta\to \eta$ is a homotopy equivalence between the total spaces, then $g_*$ is a homotopy equivalence.
\item[(b)] if $f\colon B' \to B$ is a homotopy equivalence, then so is $f^*$.
\end{enumerate}
\end{defn}

\begin{prop}\label{lem:homotopy_invariance} 
The bivariant theory $B\calC_\bullet$ is homotopy invariant.
\end{prop}
\begin{proof}
(a) As the push-forward operation is simplicial by (ii) above, it sends fiberwise homotopies to homotopies and hence 
fiberwise homotopy equivalences to homotopy equivalences. But it is well-known that $g: \theta \to \eta$ is automatically 
a fiberwise homotopy equivalence if it is a homotopy equivalence on total spaces.

(b) By naturality and the 2-out-of-6 property, it is enough to show that if $f$ is homotopic to an identity map, 
then $f^*$ is a homotopy equivalence. Again by naturality, it is enough to show that the endpoint inclusions 
$k_i\colon B\times \{i\}\to B\times [0,1]$, for $i=0,1$, induce homotopy equivalences $(k_i)^*$ 
for a given fiberwise tangential structure $\eta$ over $B\times [0,1]$. But fibrations over $BO(d)\times B\times [0,1]$ 
are fiber homotopy equivalent to product fibrations, so by part (a), we may assume that the fiberwise tangential structure 
over $B\times [0,1]$ is pulled back from a structure $\theta$ over $B$ via the projection map $B\times [0,1]\to B$. 

Note that $k_i\colon B\to B\times [0,1]$ is a homotopy equivalence \emph{over $B$}. It then follows, using the simplicial 
enrichment explained in (i) above, that $(k_i)^*$ is a homotopy equivalence.
\end{proof}

There is a canonical transformation of bivariant \addws{space-valued theories 
$$\mathcal{H}: B \calC \to B \calC_{\bullet},$$
i.e., a collection of continuous maps $B\calC(\theta)\to B \calC(\theta)_{\bullet}$ compatible} 
with both the push-forward and the pull-back operations, which is given by the inclusion of the $0$-skeleta. 

This bivariant transformation 
is universal among transformations from the bivariant theory $B\calC$ to a homotopy invariant bivariant theory: if $F$ is a bivariant theory which is homotopy invariant, 
and $\mathcal{F} \colon B\calC\to F$ is a bivariant transformation, then we obtain a natural commutative diagram of bivariant transformations
\begin{equation} \label{homotopification}
\xymatrix{
 B\calC \ar[d]^{\mathcal{H}} \ar[rr]^{\mathcal{F}} && F \ar[d]^\simeq\\
 B\calC_\bullet  \ar[rr]^{\vert \mathcal{F}_\bullet \vert} && \vert  F_\bullet\vert
}
\end{equation}
where the vertical maps are the inclusions of the $0$-skeleta to the fat geometric realizations. Since the theory $F$ is homotopy invariant, it follows 
that the right vertical map is a natural weak equivalence. The canonicity of this factorization also implies that the factorization of $\mathcal{F}$ through 
$\mathcal{H}$ is unique when regarded in the appropriate homotopy category. 



\medskip 

The bivariant theory $B\calC_\bullet$ actually takes values in the category of infinite loop spaces. This can be seen by noticing that the partial monoidal structure given 
by union of subsets (whenever this is well-defined) gives rise to the structure of a (special) $\Gamma$-space in the sense of Segal \cite{Segal(1974)}. 
Moreover, it is easy to see that it is also group-like. Following \cite{Nguyen(2015)}, the $\Gamma$-space structure can be described more precisely by varying the tangential $\theta$-structure: the value of the $\Gamma$-space at the pointed set $n_+$ is 
$$B \calC(\theta(n_+))_{\bullet}$$
where the parametrized tangential structure is the projection 
$$\theta(n_+):  \coprod_n X \xrightarrow{\coprod_n \theta} BO(d) \times B.$$
To see that this satisfies the Segal condition, one compares the spaces of (unparametrized) embeddings that define the spaces of objects and of morphisms in the respective simplicial 
categories 
$$\calC(\theta(n_+))_{\bullet} \ \ \text{and} \ \ \prod_n \calC(\theta)_{\bullet}.$$ 
It is easy to see that in both cases these spaces define homotopy equivalent models for the same classifying object or, alternatively, one can show directly that fiberwise 
embeddings of bundles over $B$ into $B \times \RR^{d + \infty}$ can be assumed to be disjoint, canonically up to homotopy (cf. the proof of \cite[Proposition 1]{Nguyen(2015)}). 
Proposition \ref{lem:homotopy_invariance} and the naturality of the $\Gamma$-space structure show that $B \calC_{\bullet}$ is homotopy invariant also as bivariant theory with 
values in $\Gamma$-spaces. 

\medskip 

The construction of $\calC(\theta)$ may be varied so that we allow \addws{$\theta$-structured fiber bundles of smooth compact manifolds with boundary} (embedded in \addws{ $\RR_+ \times \RR^{d-2} \times \RR^\infty$}) as objects, 
and \addws{$\theta$-structured bundles of smooth} cobordisms between these as morphisms (cf. \cite{Genauer(2008)}). We denote the corresponding \addws{\emph{parametrized $\theta$-cobordism category}} by $\calC^\dell(\theta)_{\bullet}$ and \addws{its classifying space by $B \calC^\dell(\theta)_{\bullet}$.}

\medskip

We do not know the homotopy types of $B\calC(\theta)_{\bullet}$ or $B\calC^\dell(\theta)_{\bullet}$ in general; in fact, for our purposes, it will suffice to use 
only their formal properties.

\section{The mapping property}\label{sec:mapping_property}

In this section, we will discuss a mapping property of the bivariant theory $B \calC^{\dell}_{\bullet}$ which greatly simplifies the identification of bivariant transformations out of 
this theory. This is the main result towards the proof of the Index Theorem in Section \ref{proof-of-DWW}.  

For the remainder of the paper, we will restrict our attention to bivariant theories which take values in the category of spectra. Such a bivariant theory $F$ determines for each 
fibration $\theta\colon X \to BO(d) \times B$,  

\

(i) a contravariant functor $\underline F_{\theta}$ on spaces over $B$, defined by 
\[\underline F_\theta(f): = F((\mathrm{id} \times f)^* \theta),\]

\

(ii) a covariant functor $\overline F$ on fibrations over $BO(d)$, defined by
\[ \overline F(X \to BO(d)):= F(X\to BO(d)).\]

\

\begin{rem} 
If $F$ is homotopy invariant, then clearly the functors $\underline F_{\theta}$ and $\overline F$ are homotopy invariant. 
\end{rem}

\begin{defn}
A homotopy invariant bivariant theory $F$ is called \emph{strongly excisive} if the functor $\underline F_\theta$ is strongly 
excisive for every $\theta$.
\end{defn}

We recall that $\underline F_\theta$ is called \emph{strongly excisive} if it sends homotopy colimits to homotopy limits. It is well-known that 
such a functor gives rise to a cohomology theory on spaces over $B$ with $B$-twisted coefficients given by the values of $\underline F_{\theta}$
at $\theta_b: X_b \to BO(d)$ for each $b \in B$. 

\medskip

Our main goal is to classify bivariant transformations 
\begin{equation} \label{biv_tr_cob_F}
\mathcal F\colon \Omega B\calC^\dell_\bullet \to F
\end{equation}
into bivariant theories $F$ which are strongly excisive. Such a bivariant transformation gives rise to characteristic classes of $\theta$-manifold bundles as follows.
Let $p\colon E\to B$ be a bundle of smooth compact $d$-manifolds (possibly with boundary), with a fiberwise embedding $E \subseteq B \times \RR^{\infty}$, and 
equipped with a parametrized $\theta$-structure $l\colon E\to X$ for some 
fibration $\theta\colon X\to BO(d)\times B$. \addws{Then $p$ can be regarded as a parametrized $\theta$-cobordism from $\emptyset$ to $\emptyset$ in $\calC^{\dell}(\theta)$. This defines a loop in $B\calC^\dell(\theta)_\bullet$, denoted $(p,l)$, and therefore an element in $\pi_0 \Omega B\calC^\dell(\theta)_\bullet$}. Applying $\mathcal F(\theta)$, we obtain an element
\[\mathcal F[p,l]:= [\mathcal{F}(\theta)(p, l)] \in \pi_0 F(\theta).\]
\addws{This is a characteristic class for the pair $(p,l)$ with values in the parametrized cohomology theory associated with $\underline F_{\theta}$; that is, it is natural with respect to pull-backs.} Taking the 2-skeleton of $B\calC^\dell(\theta)$ into account, one can 
show that the characteristic class $\mathcal F[p,l]$ is additive for fiberwise codimension-1-splittings. Similarly, the higher dimensional skeleta of $B\calC^\dell(\theta)_\bullet$ 
show the higher coherence of the additivity property in case of several independent codimension-1-splittings. Thus, 
the rule $(p, l) \mapsto \mathcal{F}[p,l]$ may be regarded as a \emph{coherently additive} characteristic 
class for $\theta$-manifold bundles. We will prove that a bivariant transformation \eqref{biv_tr_cob_F} is determined by the associated characteristic class evaluated at a 
specific bundle of $d$-disks. Denote by 
\[\theta^\univ\colon BO(d)^I\to BO(d)\times BO(d)\]
the evaluation at the endpoints and by
\[p_{\mathscr{D}}\colon E_{\mathscr{D}} \to BO(d)\]
the disk bundle of the universal $d$-plane bundle over $BO(d)$. We choose a fiberwise neat embedding $E_{\mathscr{D}} \subseteq BO(d) \times (0,1) \times \RR_+ \times \RR^{d-2 + \infty}$ 
\addws{over $BO(d)$} (any two are isotopic) 
and a tangential $\theta^\univ$-structure
\[
 \xymatrix{
 &&& BO(d)^I \ar[d]^{\theta^\univ} \\
E_{\mathscr{D}} \ar[rrr]^(.45){(T^vE, p_{\mathscr{D}})} \ar[rrru]^{l^{\univ}} &&& BO(d) \times BO(d)
}
\]
(any two are homotopic). These choices specify a morphism in $\calC^{\dell}(\theta^{\univ})$ (from $\varnothing$ to $\varnothing$) and this defines a loop in 
$B \calC^{\dell}(\theta^{\univ})_{\bullet}$, which we denote by $D^d_\univ$.

The classification we are aiming at will be formulated in terms of the \emph{homotopy category} of bivariant theories. Let $\mathscr{B}$ denote the category of spectrum-valued bivariant 
theories and bivariant transformations. A bivariant transformation $\mathcal F$ is a \emph{weak equivalence} if each $\mathcal F(\theta)$ is a weak equivalence (of spectra) for each 
$\theta$. The homotopy category $\mathscr{H}o(\mathscr{B})$ is obtained by formally inverting the \emph{weak equivalences} \footnote{Both the definition of the bivariant category and of the 
homotopy category may require the passage to a larger universe, see e.g.~\cite{DHKS}.}. We denote by $[F,G]$ the morphisms from 
$F$ to $G$ in this homotopy category.

\begin{thm}[Mapping property]\label{thm:mapping_property}
Let $F$ be a strongly excisive bivariant theory  with values in spectra. There is a bijection 
\begin{align*}
[\Omega B\calC^\dell_\bullet, F] &\to \pi_0 F(\theta^\univ),\\
\mathcal F & \mapsto \mathcal F[D^d_\univ].
\end{align*}
\end{thm}

Thus, a bivariant transformation out of $\Omega B\calC^\dell_\bullet$ is uniquely determined, in a homotopical sense, by a single characteristic class 
$\mathcal F[D^d_\univ]$. The proof of this theorem is given in Section \ref{coassembly}. In Section \ref{proof-of-DWW}, we deduce \addws{our main result (Theorem \ref{thm:factorization_result})} from 
Theorem \ref{thm:mapping_property}.

\begin{rem}\label{rem:space_level_mapping_property}
Actually $[\Omega B\calC_\bullet^\dell, F]$ is the set of path components of a \emph{space} of bivariant transformations $\map(\Omega B\calC_\bullet^\dell, F)$, defined through the 
hammock localization. Working more carefully (see Remark \ref{rem:stronger_mapp_prop}), our arguments will prove a space-valued strengthening of Theorem 3.3, namely, 
a weak equivalence of spaces
\begin{equation}\label{eq:space_level_map}
\map(\Omega B\calC_\bullet^\dell, F)\simeq \Omega^\infty F(\theta^\univ). 
\end{equation}
Explicitly, the map from the left to the right in \eqref{eq:space_level_map} is given by evaluation of a bivariant transformation at 
$\theta^\univ$, and then precomposition with a choice of a spectrum map $S^0\to F(\theta^\univ)$ corresponding to $\mathcal F[D^d_\univ]\in \pi_0 F(\theta^\univ)$. 
This procedure yields an element in the mapping space $\map(S^0, F(\theta^\univ))$ (of the hammock localization of the category of spectra); this mapping space is weakly 
equivalent to the right-hand side of \eqref{eq:space_level_map}.
\end{rem}

\section{Coassembly} \label{coassembly}

\subsection{Recollections} We start by recalling the concept of coassembly from \cite{Williams(2000)}. We assume that all spaces have the homotopy type of a CW complex. Let 
$F$ be a contravariant spectrum-valued (or space-valued) functor on the category of spaces over $B$. Suppose that $F$ is homotopy invariant, i.e., $F$ sends homotopy equivalences (of underlying spaces) to weak equivalences. 
Dualizing the arguments for the construction of assembly in \cite{WeWi(1995)}, one may associate functorially to such a functor $F$ a new 
functor $F^\&$, defined on spaces over $B$, which is strongly excisive.
The functor $F^\&$ comes with a natural transformation 
\[\nabla_F\colon F\to F^\&,\] 
called \emph{coassembly map}, which is a homotopy equivalence on one-point spaces
\[b\colon \{*\}\to B.\]
One can argue, following \cite{WeWi(1995)}, that these properties already characterize $F^\&$ and $\nabla_F$ up 
to a (canonical) natural weak equivalence of functors. 

Strongly excisive functors are determined 
by their restrictions to contractible spaces over $B$, i.e., those spaces which are homotopy equivalent to one-point spaces, 
because every space $f: X \to B$ is a canonical homotopy colimit of contractible spaces over $B$ given by singular simplices. This canonical homotopy colimit can be used to give an explicit model for the coassembly. In particular, the coassembly map is a natural weak equivalence if $F$ is already strongly excisive. More generally, it is the universal approximation to $F$ by a functor which is strongly excisive.

\begin{rem}\label{rem:form_of_excisive_functor}
A strongly excisive functor $F$ is naturally weakly equivalent (via a zigzag of natural transformations) to the functor
\[(f\colon X\to B) \mapsto \sections{F_B(X)}{B}\]
that sends $f: X \to B$ to the spectrum (or space) of sections of an associated fibration over 
$B$ which is obtained by applying $F$ to each (homotopy) fiber of $f$. See \cite[I.1]{Dwyer-Weiss-Williams(2003)}.
\end{rem}

\subsection{Coassembly for bivariant theories} 
In Section \ref{sec:mapping_property}, we defined strongly excisive bivariant theories. We will extend the definition of coassembly to bivariant theories and show that 
a strongly excisive bivariant theory $F$ is determined by the underlying covariant functor $\overline F$ on the category of fibrations over $BO(d)$.

\begin{defn}
A bivariant transformation $F\to G$  of homotopy invariant bivariant theories is called a \emph{bivariant coassembly map} if $G$ is strongly excisive 
 and $\overline F\to \overline G$ is a weak equivalence of covariant functors.
\end{defn}

Just as before, if $F$ is strongly excisive, then any bivariant coassembly map $F\to G$ is necessarily a weak equivalence of bivariant theories. It is not hard to see that 
any bivariant theory admits a bivariant coassembly map. Indeed, the standard construction of coassembly provides, by naturality, a bivariant natural transformation. However, 
we prefer to give a \addws{slightly different model for the bivariant coassembly map 
$$\nabla_F\colon F\to F^\&.$$ 
This model will be used to prove the following proposition.} 

\begin{prop}\label{lem:maps_into_excisive_theory}
Let  $F$ and $G$ be homotopy invariant bivariant theories and suppose that $G$ is also strongly excisive. Then the functor $F\mapsto \overline F$ induces a bijection of morphism sets
\[[F,G] \to [\overline F, \overline G]\]
in the respective homotopy categories (of bivariant theories $\mathscr{H}o(\mathscr{B})$ and 
of functors on the category of fibrations over $BO(d)$, respectively, obtained by formally inverting the weak equivalences). 
\end{prop}

The proposition says that any natural transformation $\overline F\to \overline G$ extends uniquely, in the homotopy category, to a  bivariant transformation $F\to G$. As a consequence, 
given $F$ and $G$ as in Proposition \ref{lem:maps_into_excisive_theory}, the bivariant coassembly map $\nabla_F$ induces a bijection of morphism sets in the homotopy category of bivariant
theories 
\begin{equation} \label{univ_prop_of_coassembly}
\nabla_F^*: [F^\&, G] \xrightarrow{\cong} [F, G].
\end{equation}
In other words, a bivariant transformation $F \to G$ admits an extension $F^{\&} \to G$, uniquely in the homotopy category of bivariant theories. 

\medskip 

We come to the promised definition of a bivariant coassembly map. We first describe the target $F^\&$. \addws{This will be defined in two steps.} Let $\simp (B)$ be the simplex 
category where an object is a continuous map $\sigma\colon  \Delta^n\to B$, and a monotone map $\alpha\colon [m]\to [n]$ provides a morphism $\alpha^*\sigma\to \sigma$. \addws{We 
define $\widetilde{F}^{\&}$} applied to a fibration $\theta \colon X\to BO(d) \times B$ to be the homotopy limit of the functor on the simplex category of $B$, 
\[\sigma\mapsto \overline F(\sigma^* X),\]
where $\sigma^*X$ denotes the pull-back of $\theta$ along
\[\id\times\sigma\colon BO(d)\times \Delta^n\to BO(d)\times B, \]
followed by the projection to $BO(d)$. Here we use the classical Bousfield-Kan model for the homotopy limit, that is, the (infinite loop) space of natural transformations (in $\sigma\in\simp (B)$)
\[
B(\simp B/\sigma) \to  \overline F(\sigma^* X)
\]
where we implicitly replace $F$ by an equivalent infinite loop space valued functor if necessary. (Recall that the space of  natural transformations between two space valued functors $H$ and $K$ on a category $\mathcal I$ is defined as the equalizer of the canonical diagram of mapping spaces
\[\prod_{i\in \mathcal I} \map(H(i), K(i)) \rightrightarrows \prod_{i\to j\in \mathcal I} \map(H(i), K(j));\]
the homotopy limit defined in this way inherits the structure of an infinite loop space if $K$ takes values in infinite loop spaces.
We refer the reader to \cite{Hi} for 
a detailed account of the properties of this construction in general (simplicial) model categories.)

Then we define $F^{\&}$ to be a simplicial thickening of the bivariant theory $\widetilde{F}^{\&}$ defined by
$$F^{\&}\fibr{X}{BO(d) \times B} \colon = \holim_{\simp(B)}  \bigg\vert F\fibr{\sigma^* X \times \Delta^{\bullet}}{BO(d) \times \Delta^{\bullet}}\bigg\vert$$
where $|-|$ denotes the (fat) geometric realization of a simplicial spectrum (or space). Since $\widetilde{F}^{\&}$ is already homotopy invariant, the canonical transformation $\widetilde{F}^{\&} \to F^{\&}$ is a weak equivalence of bivariant theories.

The bivariant transformation $F\to F^\&$ is constructed as follows. 
It is enough to construct a simplicial natural transformation
\[F\fibr{X}{BO(d) \times B} \times N_\bullet(\simp B/\sigma) \to F\fibr{\sigma^* X\times \Delta^\bullet}{BO(d) \times \Delta^\bullet}.\]

\noindent The image of a $k$-simplex in the nerve of $\simp B/\sigma$, that is, a sequence of maps
\[\Delta^{n_0}\to\dots \to \Delta^{n_k}\to \Delta^m\xrightarrow{\sigma} B,\]
is given as follows. Consider the last vertices of the $\Delta^{n_i}$'s in $\Delta^m$ to obtain a monotone map $[k]\to [m]$, and hence a simplicial map $\Delta^k\to \Delta^m$. 
We obtain a composite map
\[F\fibr{X}{BO(d) \times B}\to F\fibr{X\times_{BO(d) \times B} BO(d)\times \Delta^k}{BO(d) \times \Delta^k} \to F\fibr{\sigma^* X\times \Delta^k}{BO(d) \times \Delta^k},\]
using the  contravariant operation for the first map, and the covariant operation (include $X\times_{BO(d) \times B} BO(d)\times \Delta^k$ into $\sigma^*X \times \Delta^k$) for the second map. 

This provides a bivariant transformation $F\to F^\&$. \addws{Following \cite[I.1]{Dwyer-Weiss-Williams(2003)}, it can be shown that $F^\&$ is strongly excisive, and clearly 
$\overline F\to \overline{F^\&}$ is a weak equivalence.}

\begin{proof}[Proof  of Proposition \ref{lem:maps_into_excisive_theory}] 
\addws{A natural transformation $\mathcal{F}\colon \overline F\to \overline G$ induces a bivariant transformation $\widetilde{\mathcal{F}}^\&\colon \widetilde{F}^\&\to \widetilde{G}^\&$, 
since the definition of $\widetilde{F}^\&$ is natural and only depends on $\overline F$. 
We obtain a natural zigzag of bivariant transformations
\[
\xymatrix{
F \ar[r]^{\nabla_F} & F^\& \ar@{-->}[rr]^{\mathcal{F}^\&} && G^\& & G \ar[l]^{\simeq}_{\nabla_G} \\
& \widetilde{F}^{\&} \ar[u]^{\simeq} \ar[rr]^{\widetilde{\mathcal{F}}^{\&}} && \widetilde{G}^{\&} \ar[u]^{\simeq} 
}
\]
where the last map is a weak equivalence because $G$ is strongly excisive by assumption. It is not hard to see that 
this construction defines an inverse map $[\overline F, \overline G]\to  [F,G]$.}
\end{proof}

\begin{rem} \label{rem:stronger_univ_coassembly}
The construction of the coassembly associates a strongly excisive bivariant theory $G^\&$ to every homotopy invariant functor $G$ on the category of fibrations over $BO(d)$. 
Proposition \ref{lem:maps_into_excisive_theory} shows that there is a homotopy adjunction defined by the functors $F \mapsto \overline{F}$ and $G \mapsto G^{\&}$ on the 
respective categories with weak equivalences. Note that the coassembly transformation can be identified with the unit transformation of this adjunction. The canonicity of the 
zigzags in the arguments above shows that the bijection of Proposition \ref{lem:maps_into_excisive_theory} extends to a weak equivalence between the mapping spaces in the 
respective hammock localizations. 
\end{rem}

%

\subsection{The Pontryagin-Thom collapse maps} 
A concrete example of a coassembly map in our context is the \emph{parametrized Pontryagin-Thom transformation} defined on the bivariant theory $B\calC^\dell_\bullet$. For $B$ a CW complex,
this is given by natural infinite loop space maps
\begin{equation}\label{eq:transfer_map}
 \tr_{\theta} \colon B\calC^\dell(\theta)_{\bullet} \to \Omega^{\infty-1}_B \Sigma^\infty_B(X_+)
\end{equation}
where $\theta\colon X\to BO(d) \times B$ is a fibration, $\Omega^n_B$ denotes the functor of fiberwise based maps out of $B\times S^n$, $\Sigma ^n_B$ denotes the fiberwise smash product 
with $S^n$ over $B$, and $X_+:=X\amalg B$ is viewed as a fibration over $B$. The right-hand side of \eqref{eq:transfer_map} is the parametrized infinite loop space associated with the composite map 
$X \to BO(d) \times B \xrightarrow{\mathrm{proj}} B$. Since this can be identified with a space of sections over $B$, it defines a strongly excisive bivariant theory and therefore, 
for the construction of the bivariant transformation \eqref{eq:transfer_map}, it suffices to construct a bivariant transformation between the covariant parts 
$$\overline{\tr}_{\theta} \colon \overline{B \calC^{\dell}(\theta)_{\bullet}} \to \Omega^{\infty - 1} \Sigma^{\infty}(X_+).$$
Then the extension of this transformation to the parametrized setting yields the associated bivariant transformation between the corresponding strongly excisive theories. The natural 
transformation $\overline{\tr}_{\theta}$ is obtained from the weak equivalence shown by Genauer \cite{Genauer(2008)} that identified the homotopy type of the cobordism category of manifolds with boundary (see also 
\cite[6.1]{Raptis-Steimle(2014)}). This weak equivalence can be defined in terms of Pontryagin-Thom collapse maps (see also the variation used in \cite[5.3]{Raptis-Steimle(2014)}). 
We will make use of the fact that, by Atiyah duality, the Pontryagin-Thom construction also provides a model for the transfer map, see, e.g., \cite[I.5]{Dwyer-Weiss-Williams(2003)}, \cite{Becker-Gottlieb(1976)}, \cite[3.12]{Crabb-James}. Then the 
extension of $\overline{\tr}_{\theta}$ to the corresponding strongly excisive bivariant theories is the weak equivalence of bivariant theories which, by definition, is given by 
parametrized transfer maps. 

\medskip

For the sake of completeness, we will give a more direct definition of $\eqref{eq:transfer_map}$ as a (zigzag of) natural transformation(s) of bivariant theories with values in $\Gamma$-spaces.  
This is a straightforward generalization of the classical Pontryagin-Thom collapse map to a parametrized setting. 

We first consider the case where $B$ is a compact ENR (e.g.\ a compact CW complex). In this case, we may assume that the objects and morphisms 
of the parametrized $\theta$-cobordism category $\calC^{\dell}(\theta)$ are embedded in $B \times \RR^{d + N}$, for some $N > 0$ which is not part of the structure. 
The category $\calC^{\dell}(\theta)$ contains a subcategory $\calC^{\dell, 1}(\theta)_N$ which consists of those objects whose underlying 
manifold bundles are so that 
\[
E\subset B\times \{a\} \times \RR_+ \times \RR^{d - 2 + N}
\]
has the property that any point in the open 1-neighborhood of any fiber $E_b\subset \RR_+\times \RR^{d-2+N}\subset \RR^{d-1+N}$ has a unique closest point  in $E_b$.
The definition of the morphisms is similar. For such an object $E$, 
there is a continuous projection 
$$\pi\colon E^1\to E$$ 
from the open fiberwise 1-neighborhood of $E$ in $B\times \RR^{d - 1 + N}$ to its closest point in the corresponding fiber of $E$. 

The map $\eqref{eq:transfer_map}$ is induced by a functor to the Moore path category of the space in the target. We recall that given a space $X$, the \textit{Moore path category} 
$\Path(X)$ is a topological category whose space of objects is $X \times \RR$ and a morphism from $(x_1, a_1)$ to $(x_2, a_2)$, $a_1 < a_2$, is given by a continuous map 
$\gamma: \RR \to X$ such that $\gamma(t) = x_1$ for $t \leq a_1$ and $\gamma(t) = x_2$ for $t \geq a_2$. The topologies on the spaces of objects and morphisms 
are induced from the topology on $X$. The inclusion of objects $X \xrightarrow{\sim} B \Path(X)$ is a weak equivalence. 

Then we define a functor
\[\tr_{\theta, N} \colon \calC^{\dell, 1}(\theta)_N \to \Path(\Omega_B^{d-1+ N} \Sigma^{d+ N}_B(X_+))\]
(cf. \cite[\S 1]{GMTW(2009)}) by sending a morphism $(E,p, a_0, a_1, l)$ to the map
\begin{align*}
B \times (S^{d-1+N} \wedge [a_0, a_1]_+) &  \to (D^{d+N} \times X)/{\sim}, \\
\quad (b, x)  & \mapsto
 \begin{cases}
  [x-\pi(x), l\circ \pi(x)], &\mathrm{if}~x\in E^1,\\
  \infty, &\mathrm{otherwise,}
 \end{cases}
\end{align*}
where we identify $S^{d-1+N}\wedge [a_0, a_1]_+$ with the one-point compactification of $\RR^{d-1+N}\times [a_0, a_1]$. 

Taking classifying spaces and using the canonical equivalence $X \simeq B\Path(X)$, we obtain a (zigzag of) bivariant transformations (with the same notation)
$$
\tr_{\theta, N} \colon B\calC^{\dell, 1} (\theta)_N \to \Omega_B^{d-1+ N} \Sigma^{d+ N}_B(X_+). 
$$
These maps are compatible with stabilization in $N$, up to a canonical homotopy which is natural in $\theta$, so we obtain in the limit an induced bivariant transformation 
\begin{equation} \label{PT_stab}
\tr_{\theta} \colon B\calC^{\dell, 1} (\theta) \to \Omega_B^{\infty-1} \Sigma^{\infty}_B(X_+). 
\end{equation}
Since the target of \eqref{PT_stab} is homotopy invariant, this bivariant transformation extends canonically to the simplicial thickening, \addws{in the sense of a canonical factorization as in Diagram \eqref{homotopification} of Section 2, (again, we use the same notation)}
\begin{equation} \label{PT_htpy_ext} 
\tr_{\theta}: B\calC^{\dell, 1}(\theta)_\bullet \to \Omega_B^{\infty-1} \Sigma^{\infty}_B(X_+). 
\end{equation}
Lastly, the inclusion of simplicial categories
\begin{equation} \label{def_inj_rad} 
\calC^{\dell, 1}(\theta)_{\bullet} \to \calC^{\dell}(\theta)_\bullet
\end{equation}                                                                                                 
induces a homotopy equivalence on classifying spaces by using a deformation retraction that ``stretches'' the embedded bundle. Thus, by combining \eqref{PT_htpy_ext} and \eqref{def_inj_rad}, we obtain the required bivariant 
transformation \eqref{eq:transfer_map}. This is natural in $\theta$ and therefore extends to the respective $\Gamma$-spaces. 

This explicit description of $\tr$ does not immediately apply when $B$ is not compact because in this case manifold bundles 
will not necessarily admit an embedding in $B \times \RR^{d + N}$. But we can extend the transformation by a formal argument as follows. We may assume that $B$ is a 
CW complex and therefore it is the colimit of the diagram of its compact subspaces, $U: I_B \to \mathrm{(compact \ spaces)}$, where $I_B$ denotes the poset of compact 
subspaces of $B$ ordered by inclusion. 
For $K \in \ \mathrm{Ob} I_B$, let $\theta_{|K}$ be the pullback of $\theta$ associated with the inclusion $K \subset B$. Then the contravariant 
functoriality of the parametrized $\theta$-cobordism category yields functors, natural in $K \in I_B$, 
$$\calC^{\dell}(\theta) \to \calC^{\dell}(\theta_{|K}).$$
Thus, passing to the simplicial thickenings and the geometric realizations, we obtain $\tr_{\theta}$ as a natural (zigzag) map
$$
B \calC^{\dell}(\theta)_{\bullet} \to B( \lim_{I_B} \ \calC^{\dell}(\theta_{| K})_{\bullet})  \to \holim_{I_B} B\calC^{\dell}(\theta_{| K})_{\bullet} \to 
\holim_{I_B} \Omega_K^{\infty-1} \Sigma^{\infty}_K(X_{|K}{}_+)
$$
where the last map is obtained from the functors constructed above, the middle map from the canonical map from a limit to a 
homotopy limit, and note that 
$$\Omega_B^{\infty-1} \Sigma^{\infty}_B(X_+) \xrightarrow{\sim} \holim_{I_B} \Omega_K^{\infty-1} \Sigma^{\infty}_K(X_{|K}{}_+)$$
This definition is natural in $B$ and in $\theta$. 

\begin{thm} \label{transfer=coassembly}
The bivariant transformation \eqref{eq:transfer_map} is a coassembly map.
\end{thm}

\begin{proof}
It is clear that the target of \eqref{eq:transfer_map} is strongly excisive as a bivariant theory defined for fibrations of the form $X \to B$. Moreover, when $B$ is a point, and 
$\theta\colon X\to  BO(d) \times \{*\}$ is a fibration,  then $\calC^{\dell}_\bullet$ is a simplicial model for the (unparametrized) $\theta$-cobordism category of manifolds with 
boundaries. In this case, $\tr_{\theta}$ is a weak equivalence by \addws{a theorem of} Genauer \cite{Genauer(2008)}.
\end{proof}

\begin{ex*}
Let $d=0$ and $\theta$ the identity map on $B$. Then there is essentially only one object in $\calC^{\dell}(\theta)$, namely the empty set, and the morphisms are given 
by finite covering spaces over $B$. From covering space theory, it is known that this depends only on the fundamental group of $B$; indeed $\Omega B\calC_\bullet(\mathrm{id}_B)$ is weakly equivalent to the $K$-theory spectrum of the category of finite sets with a $\pi_1(B)$-action, provided $B$ is path-connected. It follows that $\Omega  B\calC_\bullet(\mathrm{id}_{(-)})$ is not excisive and therefore the coassembly map is not a weak equivalence. (Note that $\calC^\dell(\theta)=\calC(\theta)$ in dimension 0.)
\end{ex*}

\subsection{Proof of Theorem \ref{thm:mapping_property}}

By Proposition \ref{lem:maps_into_excisive_theory}, $\mathcal{F}$ is determined uniquely by its values for $B$ a point. In this case, the \addws{parametrized Pontryagin-Thom} collapse map
\[\tr_{\theta} \colon \Omega B\calC^{\dell}_{\bullet}\fibr{X}{BO(d)}\xrightarrow{\simeq} Q(X_+)\]
is a weak equivalence (of infinite loop spaces) by \cite{Genauer(2008)}. Hence $\mathcal{F}$ is determined by the induced transformation of covariant functors 
with values in infinite loop spaces 
$$\overline{\mathcal{F}}\colon Q \to \overline{F}$$
where $Q\colon (X \xrightarrow{\theta} BO(d)) \mapsto Q(X_+)$. The transformation $\overline{\mathcal{F}}$ is determined uniquely by the natural transformation of space-valued functors
$$\overline{\mathcal{F}}_0\colon I \to \overline{F}$$
where $I: (X \to BO(d)) \mapsto X$.

For $\sigma\colon \Delta^n\to BO(d)$ we denote by $\sigma^{\fib}$ the associated fibration over $BO(d)$ replacing $\sigma$. A concrete model for this is the pull-back
\[\xymatrix{
\sigma^{\fib} \ar[rr] \ar[d] && BO(d)^I \ar[d]^{\theta^\univ}\\
BO(d)\times \Delta^n \ar[rr]^{\id_{BO(d)\times \sigma}}& & BO(d)\times BO(d)
}\]
together with the projection to $BO(d)$. These maps $\sigma^{\fib}\to BO(d)$ induce a homotopy equivalence
\[\hocolim_{\sigma\in\simp BO(d)} \sigma^{\fib} \to   BO(d)\]
which fits into a commutative square
\[\xymatrix{
{\hocolim_{\sigma} \sigma^{\fib}} \ar[d]^\simeq \ar[rr]^{\overline{\mathcal{F}}_0(\sigma^{\fib})} 
 && {\hocolim_{\sigma} F(\sigma^{\fib})} \ar[d]
\\
BO(d) \ar[rr]^{\overline{\mathcal{F}}_0(\mathrm{id}_{BO(d)})} && F(\mathrm{id}_{BO(d)})
}\] 
Since the functor $I$ is strongly excisive, \addws{as covariant functor on the category of spaces over $BO(d)$}, it follows similarly that the natural transformation $\overline{\mathcal{F}}_0$ 
is determined uniquely, up to homotopy, by its restriction to the objects $(\sigma^{\fib} \to BO(d))$,
\[
\overline{\mathcal{F}}_{0, \Delta}(\sigma^{\fib} \to BO(d)) \colon \sigma^{\fib} \to F(\sigma^{\fib}),\]
regarded as a natural transformation of functors on the simplex category of $BO(d)$.

Note that the domain of this restricted natural transformation is a functor that takes values in contractible spaces. \addws{It follows that such natural transformations give rise to elements in the homotopy limit of the target functor:} passing to the homotopy limits, we obtain a map of spaces
\[\holim(\overline{\mathcal{F}}_{0, \Delta})\colon \holim_{\sigma\in \simp BO(d)} \sigma^{\fib} \to \holim_{\sigma \in\simp BO(d)} F(\sigma^{\fib}).\]
The domain of this map is contractible, so its homotopy class is determined by an element  
\begin{equation}\label{eq:element_x_f}
x_{\mathcal{F}} \in \pi_0 \holim_{\sigma\in \simp BO(d)} F(\sigma^{\fib}).
\end{equation}
\addws{Moreover the natural transformation $\overline{\mathcal{F}}_{0,\Delta}$ (and therefore $\mathcal{F}$ itself) is 
determined uniquely, in the homotopy category, by this element. This follows from the description of the homotopy limit as the derived mapping space out of the constant contractible diagram.}

Note that the homotopy limit in \eqref{eq:element_x_f} is by definition the target of the coassembly map for $F$ applied to $\theta^\univ$. Since $F$ is strongly excisive, 
we conclude that $x_{\mathcal{F}}$ lifts uniquely along the coassembly map to an element
\[y_{\mathcal{F}} \in \pi_0 F(\theta^\univ).\]

Hence the natural transformation $\mathcal{F}$ is determined, uniquely in the homotopy category, by the element $y_{\mathcal{F}}$. Then it remains to show that 
$y_{\mathcal{F}} =\mathcal{F}[D^d_\univ]$. To do this, consider the commutative \addws{diagram}:
\[\xymatrix{
\pi_0 \Omega B\calC^\dell(\theta^\univ)_{\bullet} \ar[rr]^(.55){\mathcal{F}} \ar[d]^{\nabla}
 && \pi_0 F(\theta^\univ) \ar[d]^{\nabla}_\cong
\\
\pi_0 \holim_\sigma \Omega B\calC^\dell(\sigma^{\fib})_{\bullet} \ar[rr]^{\mathcal{F}} \ar[d]_\cong^{\tr}
 && \pi_0 \holim_\sigma F(\sigma^{\fib})
\\
 \pi_0 \holim_\sigma Q_+(\sigma^{\fib}) \ar[rru]^{\overline{\mathcal{F}}}  \ar[d]_\cong
&& \pi_0 \holim_\sigma (\sigma^{\fib}) \ar@{_{(}->}[ll]  \ar[u]_{\overline{\mathcal{F}}_{0, \Delta}} \ar@{=}[d]
\\
 [BO(d), QS^0] && \ast \ar@{_{(}->}[ll]
}\]
where the bottom horizontal map is the inclusion of the constant map at $1$. By construction, $y_{\mathcal{F}}$ is the image of 
$\ast$ in $\pi_0 F(\theta^\univ)$ under the composite map in the right column. 

Thus, it remains to show that under the left vertical composite, $D^d_\univ$ maps to the constant homotopy class $BO(d)\to QS^0$ at $1$. But 
this class is obtained from the usual parametrized transfer of the universal disk bundle $p_{\mathscr{D}}$, \addws{using the identifications 
$E_{\mathscr{D}} \stackrel{l^{\univ}}{\simeq} BO(d)^I \simeq BO(d)$ over $BO(d)$. Therefore $D^d_{\univ}$ indeed maps to the constant homotopy 
class at $1$ by the homotopy invariance of the transfer map.} \qed

\begin{rem} \label{rem:stronger_mapp_prop}
One can prove the space-level  version of Theorem \ref{thm:mapping_property} (see Remark \ref{rem:space_level_mapping_property}) along the same lines, replacing 
homotopy classes of maps by mapping spaces throughout. A more concise argument can be given using an iterative application of Theorem \ref{thm:mapping_property} 
as follows. Let $F$ be strongly excisive and denote by $\Omega F$ the bivariant theory obtained by applying the (derived) loop functor of spectra at every $\theta$. 
There is a homotopy commutative  square 
\[\xymatrix{
 \map(\Omega B\calC_\bullet^\dell, \Omega^n F) \ar[r]\ar[d]^\simeq
 & \Omega^n \map(\Omega B\calC_\bullet^\dell, F) \ar[d]^\simeq
 \\
 \map(\overline{\Omega B\calC_\bullet^\dell}, \Omega^n \overline{F}) \ar[r]^\simeq
 &  \Omega^n \map( \overline{\Omega B\calC_\bullet^\dell}, \overline{F})
 }\]
where the vertical maps are given by applying the  functor $F\mapsto \overline F$; these maps are weak equivalences by a space-level version of Proposition 
\ref{lem:maps_into_excisive_theory} (see Remark \ref{rem:stronger_univ_coassembly}). The lower horizontal map is a weak equivalence because one knows, in this case,
that the mapping spaces in the hammock localization are equivalent to the derived mapping spaces. We conclude that  the upper horizontal arrow is also a weak 
equivalence. This implies 
\[
\pi_n \bigr(\map(\Omega B\calC_\bullet^\dell, F)\bigl) \cong [\Omega B\calC_\bullet^\dell, \Omega^n F] 
\]
and therefore the space-level version of the mapping property follows from Theorem \ref{thm:mapping_property}.
\end{rem}

\section{Proof of the Index Theorem} \label{proof-of-DWW}

\subsection{The map to $A$-theory}\label{A-theory}

Recall from \cite{Boekstedt-Madsen(2014), Raptis-Steimle(2014), Tillmann(1999)} that there is a map from the loop space 
of the classifying space of the standard $d$-dimensional cobordism category (with or without boundaries) to $A(BO(d))$, Waldhausen's algebraic $K$-theory of the space $BO(d)$. Using the bivariant extension of $A$-theory \cite{Williams(2000)} (see \cite[Section 3]{Raptis-Steimle(2014)} for a detailed discussion), it is easy to extend  
the definition of this map fiberwise to parametrized cobordism categories and also allow arbitrary parametrized 
$\theta$-structures (see also \cite[6.1]{Raptis-Steimle(2014)}).

We start by defining a bivariant transformation to bivariant $A$-theory 
\begin{equation} \label{pre_tau_map} \tau(\theta)\colon \Omega B \calC^\dell (\theta) \to A\fibr{X}{B}.\end{equation}

\begin{rem} \label{comparison}
We recall that bivariant $A$-theory is defined for general fibrations $E \to B$. Since a fibration $\theta: X \to BO(d)\times B$ yields, by projection, a fibration over $B$, we may view the 
correspondence $\theta \mapsto A(X \to B)$ as a bivariant theory in our sense.
\end{rem}

The bivariant transformation \eqref{pre_tau_map} is obtained from the geometric realization of a simplicial map
$$N_{\bullet} \calC^\dell(\theta) \to  w S_{\bullet} \Rfd\fibr{X}{B}.$$
This is defined on $k$-simplices as follows: a $k$-simplex in the domain is a bundle of cobordisms $E[a_0, a_k]$ over $B$, which is the union of a sequence of 
composable bundles of cobordisms, 
$$E[a_0, a_1], \ E[a_1, a_2], \ \cdots, \ E[a_{k-1}, a_k],$$
and it is endowed with a fiberwise tangential $\theta$-structure. It is mapped  to the diagram of retractive spaces 
$$\mathrm{Ar}[k] \to \Rfd\fibr{X}{B}$$
$$(i \leq j) \mapsto E[a_i, a_j] \cup_{E(a_i)} X $$
where 
$$E(a_i) = E[a_0, a_k] \cap (B \times \{a_i\} \times \RR^{d-1 + \infty})$$
$$E[a_i, a_j] = E[a_0, a_k] \cap (B \times [a_i, a_j] \times \RR^{d-1 + \infty})$$
and the pushout is defined using the tangential structure restricted to $E(a_i)$. This simplicial map induces the bivariant transformation $\tau(\theta)$ in \eqref{pre_tau_map}.
As bivariant $A$-theory is homotopy invariant, it follows that $\tau(\theta)$ extends, canonically in the homotopy category, to a bivariant transformation 
(for which we use the same notation)
\begin{equation} \label{tau_map}
\tau(\theta)\colon \Omega B \calC^\dell (\theta)_\bullet \to A\fibr{X}{B}.
\end{equation}

\begin{rem*}
The simplicial thickening of bivariant $A$-theory may be called the \emph{thick model} of bivariant $A$-theory (cf. \cite{Raptis-Steimle(2014)}). \addws{Strictly speaking,} $\tau(\theta)$ in \eqref{tau_map} is a map to the thick model of bivariant $A$-theory.
\end{rem*}

The infinite loop space structure of bivariant $A$-theory can also be described in terms of a $\Gamma$-space (or $\Gamma$-category) by varying the fibration. 
The value of the $\Gamma$-space for $A(X \to B)$ at $n_+$ is 
$$A\fibr{\coprod_n X}{B}.$$
Since $\tau(\theta)$ is natural with respect to the tangential structure, it is easy to see that it extends to a map of $\Gamma$-spaces and therefore is an infinite loop map.


\subsection{\addws{Comparing bivariant transformations}}

Let $\nabla_A\colon A\to A^\&$ denote the coassembly map for bivariant $A$-theory. We recall that the unit map to $A$-theory is the natural transformation of covariant functors 
\[\eta(Z) \colon Q(Z_+)\to A(Z)\]
that is characterized by the property that it takes values in \addws{the category of} infinite loop spaces and that, for $Z$ 
the point, it sends $1\in \pi_0 QS^0$ to $1\in \pi_0 A(*)$. By Proposition \ref{lem:maps_into_excisive_theory} the unit map extends to a bivariant transformation 
\[\eta(\theta) \colon \Omega^\infty_B \Sigma^\infty_B(X_+) \to A^\&\fibr{X}{B},\]
uniquely in the homotopy category of bivariant theories on fibrations $\theta: X \to B$. \addws{As explained in Remark
\ref{comparison}, this can also be regarded as a bivariant transformation between bivariant theories defined on fibrations 
$X \to BO(d) \times B$, for fixed $d$, by using the projection onto $B$.} 


\begin{thm}\label{thm:factorization_result}
Let $\theta: X \to BO(d) \times B$ be a parametrized tangential structure and $\overline \theta : X \to B$ denote the composition with the projection. Then the following diagram commutes in the homotopy category of infinite loop spaces:
\[\xymatrix{
 \Omega B\calC^\dell(\theta)_{\bullet} \ar[d]_(.4){\tau(\theta)} \ar[rr]^{\tr_{\theta}}
 && \Omega^\infty_B \Sigma^\infty_B(X_+) \ar[d]^(.4){\eta(\overline \theta)}
\\
A\fibr{X}{B} \ar[rr]^{\nabla_{A}(\overline \theta)} 
 && A^\&\fibr{X}{B}
}\]
\addws{Moreover, the corresponding bivariant transformations $\eta \circ \tr$ and $\nabla_A \circ \tau$ agree in the homotopy category of bivariant theories on fibrations over $BO(d) \times B$.}
\end{thm}

\begin{proof}
We show that the corresponding bivariant transformations define a commutative diagram in the homotopy category of bivariant theories on fibrations over $BO(d) \times B$. 
To do this, by Theorem \ref{thm:mapping_property}, \addws{it suffices to show that the two characteristic classes, corresponding to the two possible composites, 
agree on the universal $d$-disk bundle $D^d_\univ$, i.e., 
$$(\eta \circ \tr)[D^d_{\univ}] = (\nabla_A \circ \tau)[D^d_{\univ}] \ \text{in} \ \pi_0 \big( A^{\&}\fibr{BO(d)^I}{BO(d)} \big),$$
where $D^d_{\univ}$ is regarded as a point in $\Omega B \calC^{\dell}(\theta^{\univ})_{\bullet}$. The computation of these two classes is straightforward and 
can be done directly by making use of the fiberwise homotopy invariance of the two constructions in this case. We will also argue this way that it is enough 
to show that the two characteristic classes agree on the $d$-disk $D^d$, considered as a bundle over a point.

As was already noted in the proof of Theorem \ref{thm:mapping_property}, the characteristic class obtained from the bivariant transformation $\tr$,
$$\tr[D^d_{\univ}] \in \pi_0( \Omega^{\infty}_{BO(d)} \Sigma^{\infty}_{BO(d)} BO(d)^I_+),$$ 
is given by the parametrized transfer of $p_{\mathscr{D}} \colon E_{\mathscr{D}} \to BO(d)$ followed by the map induced by the parametrized tangential structure $l^{\univ}$ to $BO(d)^I$. Using that $BO(d)^I \stackrel{\simeq}{\to} BO(d)$, this class can be identified with the parametrized
transfer of $p_{\mathscr{D}}$ followed by the map induced by the homotopy equivalence $p_{\mathscr{D}}$. Since the parametrized transfer is invariant under fiberwise homotopy 
equivalences, it follows that $\tr[D^d_{\univ}]$ can be identified in $\pi_0(\Omega^{\infty}_{BO(d)} \Sigma^{\infty}_{BO(d)} BO(d)_+)$ with the class that arises from the 
parametrized transfer of the trivial $d$-disk bundle over $BO(d)$, which is pulled back from the trivial $D^d$-bundle over a point. 

Similarly, the characteristic class obtained from $\tau$,
$$\tau[D^d_{\univ}] \in \pi_0 A\fibr{BO(d)^I}{BO(d)}$$
is represented by the object in $\Rfd\fibr{BO(d)^I}{BO(d)}$ given by 
$$E_{\mathscr{D}} \sqcup BO(d)^I \to BO(d),$$
as retractive object over $BO(d)^I \to BO(d)$. Using again that $BO(d)^I \stackrel{\simeq}{\to} BO(d)$ and $E_{\mathscr{D}} \simeq BO(d) \times D^d$ over $BO(d)$, 
this class agrees in $\pi_0 A\fibr{BO(d)}{BO(d)}$ with the class represented by the retractive object
$$BO(d) \times D^d \sqcup BO(d) \to BO(d),$$
where $BO(d) \times D^d \to BO(d)$ is the trivial $D^d$-bundle, which is pulled back from the trivial $D^d$-bundle over a point. 

It follows by naturality that both $\tr[D^d_\univ]$ and $\tau[D^d_\univ]$ are determined canonically by the corresponding classes for the $d$-disk $D^d$ considered 
as a bundle over a point and equipped with its canonical framing. Thus, we only need to show that
\[\eta(\tr[D^d]) = \tau[D^d] \ \text{in} \ \pi_0 A(*) \cong \mathbb Z.\]
But both characteristic elements are easily identified with $1 \in \mathbb Z$. }
\end{proof}

\begin{rem}
We could alternatively appeal to Proposition \ref{lem:maps_into_excisive_theory} and reduce to $B=\{*\}$, in which case the claim was proved in 
\cite{Raptis-Steimle(2014)}. However the use of infinite loop space techniques and the covariant functoriality drastically simplify the proof. 
\end{rem}

\begin{proof}[\textbf{Proof of the Index Theorem}]
Let $p: E \to B$ be a bundle of smooth compact $d$-dimensional manifolds. We may assume that $B$ is a CW complex. Then there is a fiberwise smooth embedding of $p$ into $B \times \mathbb{R}^{\infty}$. 
We denote $T^v E$ the Gau\ss{} map of the vertical 
tangent bundle and factor the map
\[(T^v E,p)\colon E\to BO(d)\times B\]
into a homotopy equivalence $l\colon E\to X$, followed by a fibration $\theta\colon X\to BO(d)\times B$. This 
way we obtain a morphism in $\calC^{\dell}(\theta)$ (from $\varnothing$ to $\varnothing$), and therefore an 
element in $\Omega B \calC^{\dell}(\theta)$ which we denote $(E, l)$. 

Now the class 
\[\tau[E, l]\in \addws{\pi_0}A\fibr{X}{B}\]
is given by the retractive space $E \sqcup X$ over $X$. This object can be identified canonically, via the homotopy equivalence $l$, with the \emph{bivariant $A$-theory characteristic} $\chi(p)\in A(p)$ of $p$ which is given by the  
retractive space $E\times S^0$ over $E$. It was observed in \cite{Williams(2000)} that the parametrized $A$-theory characteristic of $p$ is the image of that element 
under the coassembly map (cf. \cite{Raptis-Steimle(2014)}). 
Then the claim follows directly from Theorem \ref{thm:factorization_result}, after noting \addws{that  
$$\tr[E, l] \in \pi_0(\Omega^{\infty}_B \Sigma^{\infty}_B (X_+))$$} 
is indeed, along the identification $l$, the parametrized transfer of the  bundle $p$. 
\end{proof}

\begin{rem}
A finer statement can be made if we view $\chi(p)$ and $\eta\tr(p)$ as points in $\Omega^\infty A^\&(p)$, rather than just $\pi_0$-classes. Following Remark \ref{rem:space_level_mapping_property}, the diagram in Theorem \ref{thm:factorization_result} actually commutes up to a preferred homotopy (obtained from the preferred path between $\eta(\tr[D^d])$ and $\tau[D^d]$ which is implicit in the proof). Hence our proof actually provides a canonical (homotopy class of) paths between the two expressions in the Index theorem, rather than just identifying their $\pi_0$-classes.
\end{rem}

\begin{rem}
Waldhausen constructed a trace map
\[T\colon A(X)\to Q(X_+)\]
which is an infinite loop map that splits the unit map and is natural in $X$. This extends to a bivariant transformation $T^{\&}: A^{\&} \to Q^{\&}$ 
between the associated strongly excisive theories on fibrations $X \to B$. Then one can verify that the correspondence 
$$p \mapsto T^{\&} \Bigg(\nabla_{A,p}\bigg(\chi(p)\bigg)\Bigg)$$
satisfies the axioms from \cite[Definition 1.4]{Klein-Williams(2009)} that characterize the parametrized transfer map (see also \cite{Becker-Schultz(1998)}). 

Indeed, axiom (A1) on naturality is clear. For the normalization axiom (A2), we note that the parametrized $A$-theory characteristic $\nabla_{A, \id_B}\bigg(\chi(\id_B)\bigg)$ is 
identified with the constant map $B\to A(*)$ with value $[S^0]\in A(*)$; and by \cite[Theorem 5.1]{Waldhausen(1978)}, this element is mapped under the trace to the 
non-base point in $S^0\subset QS^0$. 

To verify the product axiom (A3), we recall that $A$-theory comes with an external pairing
\[A(X)\times  A(X')\to A(X\times X'),\]
given by applying $K$-theory to the external smash  product functor on the level of retractive spaces. This external pairing is natural in both variables 
so we get an induced pairing
\[\sections{A_B(E)}{B}\times \sections{A_{B'}(E')}{B'}\to \sections{A_{B\times  B'} (E\times E')}{B\times B'}\]
for two given fibrations $p\colon E\to B$ and $p'\colon E'\to B'$; under this construction the elements $\nabla_{A,p}(\chi(p))$ and $\nabla_{A,p'}(\chi(p'))$ (when defined) 
pair to the element 
$$\nabla_{A, p \times p'}(\chi(p\times p')).$$ 
The claim follows from this because Waldhausen's trace is multiplicative \cite[Proposition 3.7]{Dorabiala-Johnson(2012)}.

Finally, to verify the additivity axiom (A4), we note that the bivariant characteristics of the fibrations in \addws{axiom (A4)} of \cite{Klein-Williams(2009)} 
satisfy
\[\chi(p) = (j_1)_*\chi(p_1) + (j_2)_*\chi(p_2) - (j_\emptyset)_*\chi(p_\emptyset).\]
The required identity follows from this because both coassembly and trace are infinite loop maps and in particular additive.
\end{rem}


\bibliographystyle{amsplain}
\providecommand{\MR}{\relax\ifhmode\unskip\space\fi MR }
\providecommand{\MRhref}[2]{%
  \href{http://www.ams.org/mathscinet-getitem?mr=#1}{#2}
}
\providecommand{\href}[2]{#2}

\end{document}